\documentclass[12pt]{article}
\usepackage{amsmath,amsfonts,amssymb,amsthm}
\usepackage{xspace}
\usepackage{xr}
\usepackage{float}
\usepackage{makeidx}
\usepackage{graphicx}
\usepackage{epsfig} 
\usepackage{url}

\newcommand{\reseteqn}{\setcounter{equation}{0}}

\newcommand{\defas}{\mathrel{\raise.095ex\hbox{:}\mkern-4.2mu=}}

\newcommand{\bdm}{\begin{displaymath}}
\newcommand{\edm}{\end{displaymath}}

\newcommand{\bP}{\ensuremath{\mathbf{P}}\xspace}
\newcommand{\bE}{\ensuremath{\mathbf{E}}\xspace}

\newcommand{\bd}{\ensuremath{\mathbf{d}}\xspace}

\newcommand{\bR}{\ensuremath{\mathbb{R}}\xspace}

\newcommand{\bZ}{\ensuremath{\mathbb{Z}}\xspace}

\newcommand{\cE}{\ensuremath{{\cal E}}\xspace}

\newcommand{\cF}{\ensuremath{{\cal F}}\xspace}

\renewcommand{\leq}{\ensuremath{\leqslant}}
\renewcommand{\geq}{\ensuremath{\geqslant}}

\newcommand{\gq}{\ensuremath{\theta}\xspace}
\newcommand{\gt}{\ensuremath{\tau}\xspace}

\newcommand{\gl}{\ensuremath{\lambda}\xspace}

\newcommand{\gf}{\ensuremath{\varphi}\xspace}
\newcommand{\gm}{\ensuremath{\mu}\xspace}

\newcommand{\gp}{\ensuremath{\pi}\xspace}

\newcommand{\gs}{\ensuremath{\sigma}\xspace}

\newcommand{\gW}{\ensuremath{\Omega}\xspace}

\newcommand{\gj}{\ensuremath{\varphi}\xspace}

\newcommand{\mrm}{\mathrm}

\newlength{\Litil} 
\newlength{\Mid} 
\newlength{\Stor} 
\newtheorem{theorem}{Theorem}

\theoremstyle{remark}

\theoremstyle{remark}

\newcommand{\bN}{\ensuremath{\mathbb{N}}\xspace}
\newcommand{\given}{\ensuremath{\;|\;}}
\newcommand{\asdef}{\mathrel{=\mkern-4.2mu\raise.095ex\hbox{:}}}

\newcommand{\bil}{\ensuremath{\quad}}
\newcommand{\pbil}{\ensuremath{\;\;}}

\newlength{\myownsep}
\makeatother

\usepackage[all]{xy}

\begin{document}
\reseteqn

\newcommand{\ci}[1]{\ensuremath{{#1}^{\circ}}}
\newcommand{\cik}[1]{\ensuremath{{#1}^{[\circ]}}}

\title{Construction and Characterisation 
of Stationary and Mass-Stationary Random Measures on $\bR^d$}
\author{G\"unter Last\footnotemark[1]\, and Hermann Thorisson\footnotemark[2]}

\footnotetext[1]{Institute of Stochastics, Karlsruhe Institute of Technology,
Germany, guenter.last@kit.edu}

\footnotetext[2]{Department of Mathematics, University of Iceland, Iceland, hermann@hi.is}

\date{\today}

\maketitle

\vspace{-8 mm}
\begin{abstract}
\noindent
Mass-stationarity means that the origin 
is at a typical location in the mass of a random measure.
It 
 is an intrinsic characterisation of  
Palm versions with respect to stationary random measures.
Stationarity is the special case when the random measure is
Lebesgue measure.
The paper presents 
constructions of stationary and mass-stationary
versions through change of measure and 
change of origin.
Further, the paper considers 
characterisations of mass-stationarity
by distributional invariance under preserving
shifts against stationary independent backgrounds.
\end{abstract}


\section{Introduction}

\noindent
Mass-stationarity is a formalization of the intuitive idea
that the origin 
is at a typical location in the {\em mass} of a random measure;
the definition is given at (1) below.
Stationarity is the special case when the random measure is
Lebesgue measure; stationarity can be thought of as saying that the origin
is at a typical location in the {\em space}.
Mass-stationarity was introduced in \cite{LaTho09},
 where it is shown that it is a 
 characterisation of
Palm versions with respect to stationary random measures;
see Theorem~1 below.
Palm probabilities are a very important concept in 
theory and application of point processes and random measures
\cite{Thor00,Kallenberg,La10}. In stochastic geometry, for instance,
already the definition of the basic notions 
(e.g.\ typical cell, typical face)
require the use of Palm probability measures; see \cite{SW08}.
The focus of the present paper is on the intrinsic properties of these
measures.

For a simple example, consider a stationary Poisson process $N$ on the line. 
Stationarity means that shifting the origin to any location $t \in \bR$ does not 
alter the distribution of $N$; so the origin is at a typical location on the line~(in~space).
If we add an extra point at the origin 
then we obtain the mass-stationary Palm version $\ci{N} = N + \delta_0$;
the new point is 
at a typical location in the mass of $\ci{N}$
because shifting the origin to the $n^{\text{th}}$ point on the~right~(or~on~the~left) 
does not alter the fact that the inter-point distances of $\ci{N}$ are i.i.d.\! exponential. 

It is only in the Poisson case that the mass-stationary/stationary version is obtained
from the stationary/mass-stationary one by simply adding/deleting a point at the origin.
And it is only on the line that mass-stationarity of simple point processes
can be characterized by
distributional invariance under shifts of the origin to the  
$n^{\text{th}}$ point on the right (or on the left).

The aim of this paper is twofold.
We shall first 
consider constructions of stationary
and mass-stationary
versions for random measures on $\bR^d$,
and then study 
characterisations of mass-stationarity.
Actually, as required by many applications, we shall treat the random measure 
jointly with a random element, for instance a random field. 
We denote by $(X, \xi)$ a random element-and-measure which is stationary 
under a (probability) measure $\bP$, and by $(\ci{X}, \ci{\xi})$ a random element-and-measure
which is mass-stationary under a (probability) measure~$\ci{\bP}$.
We will not restrict $\bP$ and $\ci{\bP}$ to be  probability measures.
In Palm theory, this generality can in fact be quite useful
for probabilistic purposes.
For instance, two-sided Brownian motion 
is mass-stationary with respect to its local time at zero, 
but the stationary version does not have a finite distribution;
see  \cite{LaMoeTho12}.

In Section 2, \,we 
recall the definition of 
mass-stationarity, \,the definition of Palm versions,\,
and the key characterisation theorem linking
these concepts.

The construction part of the paper consists of Sections 3--5. 
In Section~3, we elaborate on the two-step change-of-measure
change-of-origin method, applied to simple point processes in 
\cite{Thor00},
to construct the mass-stationary Palm version when the stationary version
is given. 
In Section~4, we reverse this construction to obtain the
stationary version when the mass-stationary version is given.
In Section~5, we show that when the random measure 
has a density  field with respect to Lebesgue measure
then a change of origin is not needed
to construct the mass-stationary version. 
We also show that if the density  field is strictly 
positive then a change of origin 
is not needed to construct the stationary version.

The characterisation part of the paper consists of Sections 6--8. 
In Section~6, we show for random measures with
a strictly positive density field, that
mass-stationarity is characterized by
distributional invariance under preserving shifts, i.e.\!
shifts inducing allocations preserving $\xi$.
This has been known to be the case for  simple point processes on 
Abelian groups; see \cite{He:La:05,He:La:07,La10}.
This is also known to be the case for diffuse random measures on the line; 
see  Theorem 3.1 in \cite{LaMoeTho12}.
In~Section~7, we show
that the same is true for diffuse random 
measures on $\bR^{d}$ if the 
background randomization
from
\cite{Thor00} is applied.
In Section~8, we lift this shift characterization
 further to general random measures on $\bR^d$
by extending them to diffuse random measures
on $\bR^{d+1}$.

Section 9 
concludes with 
a final remark 
on 
mass-stationarity.

We end
this introduction with some
further 
background information relevant for
the topic of this paper.
Preserving allocations are a special case of
mass transports 
balancing two 
random measures;
see  \cite{HP05,LaTho09}.
Stable transports between Lebesgue 
measure
and a stationary 
point process were introduced and
studied in \cite{HP05,HHP06}. 
The algorithm of \cite{HHP06} is generalized in 
\cite{Haji-MirsadeghiKhezeli15} to balance
general stationary ergodic random measures of equal intensity.
Gravitational allocations
balancing Lebesgue measure and a stationary Poisson process were  
investigated in \cite{Ch:Pe:Per:Ro:10,Ch:Pe:Per:Ro:10a}.
Cox 
processes were used in \cite{LaTho11c} to balance Lebesgue
measure and a general diffuse random measure.
In \cite{HueSturm13}, it is shown that optimal shift-invariant transports
between Lebesgue measure and a stationary point process
exist if the average cost (defined in terms of the Palm distribution) is finite. 
In 
\cite{Hue12}, this 
is generalized
to the case of two jointly 
stationary random measures
with the first being absolutely continuous.
General transport 
formulas for random measures invariant
under group actions were derived in 
\cite{La10,Last10a,GentnerLast11,Kallenberg11}. 
In the recent paper \cite{PitmanTang15}, 
a space-time shift (inducing a balancing allocation)
is
used to find the 
Brownian bridge in the path of a Brownian motion.

\section{Preliminaries on mass-stationarity}

\noindent
Let $(\Omega, \cF)$ be the 
measurable space on which 
the random elements in this paper are defined 
(unless otherwise stated).
Let $\bP$ be a measure on $(\Omega, \cF)$.
Note that, as explained in the introduction,
we do not restrict $\bP$ to be a probability measure.

Let $\xi$ be a random measure on $\bR^d$.
For each  $t\in \bR^d$, let  $\theta_t$ be the shift map defined 
by
\begin{align*}
\theta_t \xi (B) := \xi(B+t), \quad \text{for Borel subsets $B$ of }\bR^d.
\end{align*}
Let  $\overset{D}{=}$ denote identity in distribution.
The measure $\xi$ is {\em stationary} (under $\bP$) if 
\begin{align*}
\theta_t \xi \overset{D}{=} \xi, \quad t \in \bR^d, \quad \text{(under $\bP$).}
\end{align*}
Let $\lambda$ be the Lebesgue measure on $\bR^d$.

Let $(E, \cE)$ be a measurable space 
on which the additive group $\bR^d$ acts.
For~$t\in \bR^d$, let $\theta_t$ also denote the map
taking $x$ in $E$ to $\theta_t x$ in $E$.
Let  $X$ be a random element in $(E, \cE)$.
For instance, $X$ could be a 
random field $X~=~(X_s)_{s \in \bR^d}$ and 
$\theta_t X = (X_{t+s})_{s \in \bR^d}$ for $t \in \bR^d$.
Assume that $X$ is {\em shift-measurable}, namely that
the map
from $\bR^d \times E$ to 
$E$ taking $(t, x)$ to  
$\theta_t x$ is measurable. 
Put $\theta_t (X,\xi) = (\theta_t X,\theta_t \xi)$.
The pair $(X,\xi)$ is {\em stationary} (under $\bP$) if
\vspace{-5mm}
\begin{align*}
\theta_t (X,\xi) \overset{D}{=} (X,\xi), \quad t \in \bR^d, \quad \text{(under $\bP$).}
\end{align*}
Let $\ci{\bP}$\! be another measure on $(\Omega, \cF)$,
let $\ci{\xi}$ be another random measure on $\bR^d$,
and let $\ci{X}$\! be another random element in $(E, \cE)$.
Assume that $\ci{\xi}$\! has $0$ in its support $\ci{\bP}$-a.e.
In~this~paper it is always understood
that the distributions of $(X\!, \xi)$ and $(\ci{X}\!, \ci{\xi})$
are  $\sigma$-finite under both $\bP$ and~$\ci{\bP}$.

The pair
$(\ci{X},\ci{\xi})$ is called {\em mass-stationary} (under $\ci{\bP}$) if
for all bounded Borel subsets $C$ of $\bR^d$ with $\lambda(C) > 0$
and $\lambda(
\partial C) = 0$, 
\begin{gather}\label{(1)}
(\theta_{V_C}(\ci{X},\ci{\xi}),V_C + U_C)\overset{D}{=}((\ci{X},\ci{\xi}),U_C)
\quad \text{(under $\ci{\bP}$)}
\end{gather}
\vspace{-5mm}

\noindent
where (under $\ci{\bP}$)

\vspace{-5mm}
\begin{align*}
\text{\qquad the conditional distribution of $U_C$ given  $(\ci{X},\ci{\xi})$
 is uniform on $C$, and\,}
 \end{align*}
 
 \vspace{-12mm}

 \begin{align*} 
\text{\qquad the conditional distribution of $V_C$ given $((\ci{X},\ci{\xi}),U_C)$ is $\ci{\xi}(\cdot|C-U_C)$;}
\end{align*}
here, for any $x\in \bR^d$, $C-x:=\{y-x:y\in C\}$.

For a motivation of this definition and a survey; see \cite{LaTho11b}.
In particular for a simple point process $\ci{\xi}$ on the line, this 
definition is equivalent to 
distributional invariance under shifts of the origin $n$ points forward (or backward).
And for a diffuse $\ci{\xi}$ on the line, it is equivalent to 
distributional invariance under shifts of the origin 
an amount $r$ forward (or backward) in the mass; see \cite{LaMoeTho12}.
Note that in both cases these shifts preserve the measure~$\ci{\xi}$.

Recall (see e.g.\ \cite{Kallenberg}) that
$(\ci{X},\ci{\xi})$ under $\ci{\bP}$ is called a {\em Palm version} 
of a stationary pair $(\hat{X}, \hat{\xi})$ 
defined on some $(\hat{\gW}, \hat{\cF}, \hat{\bP})$
 if for each nonnegative measurable function $f$ and some 
 (and thus each, due to stationarity)
Borel subset $B$ of  $\bR^d$ with $0 < \lambda (B) < \infty$, 
\begin{align} \label{(3.1)}
\ci{\bE}[f(\ci{X},\ci{\xi})] 
= \hat{\bE}\Big{[}\int_B f\big{(}\theta_t (\hat{X}, \hat{\xi})\big{)} \hat{\xi}(dt)\Big{]}
\Big{/}\lambda(B).
\end{align}
In this definition  $(\ci{X}\!,\ci{\xi})$ and $(\hat{X}\!, \hat{\xi})$ are allowed to
have distributions that are only $\sigma$-finite 
and not necessarily probability measures.
The measure $\ci{\bP}$ 
is finite 
if and only if 
$\hat{\xi}$ has finite intensity, that is, 
if and only if $\hat{\bE}[\hat{\xi}(B)]<\infty$ for bounded Borel $B$.
In this case $\ci{\bP}$ 
can be normalized 
to a probability measure.

For a proof of the following result, see \cite{LaTho09}.

\begin{theorem}
\label{T:0}
The pair $(\ci{X}\!, \ci{\xi})$  is mass-stationary under $\ci{\bP}$
if and only if it is the
Palm version of some stationary pair.
\end{theorem}

\section{Construction of the mass-stationary version}
\noindent 
In this 
  first construction
section, we construct the mass-stationary Palm version 
when the stationary version
is given.
Let $(X, \xi)$ be stationary under $\bP$.
We assume that 
\begin{align}\label{convexhull}
\operatorname{conv}(\operatorname{supp}\xi)=\bR^d\quad \text{$\bP$-a.e.}
\end{align}
where $\operatorname{conv}(B)$ is the convex hull of a set $B\subset\bR^d$
while $\operatorname{supp}\xi$ denotes the support of $\xi$.
This is a rather weak assumption. Indeed, if $\bP$ is a probability measure
and $\xi$ is $\bP$-a.s.\ not the null measure, then \eqref{convexhull}
holds; see Theorem 2.4.4 in \cite{SW08}. 

Let $N$ be the simple point process on $\bZ^d$ with a point at \,$i \in \bZ^d$\, 
if and only if $\xi(i + [0,1)^d) > 0$.
Consider the Voronoi cells in $\bZ^d$ obtained 
by associating each $i \in \bZ^d$ to~the point of $N$ that
is closest to $i$, choosing the one with the lowest lexicographic order 
if there are more than one such point. 
These cells contain exactly one point of $N$ and
partition $\bZ^d$  in a shift-invariant way. 
Let $D_i$ be the cell containing $i$ and
let $S_i$ be the vector from the $N$-point in $D_i$ to $i$. 
Set $D = D_0$, $S = S_0$ and $\ci{D} = S + D$.
Note that since $0 \in D$ the vector $S$ takes values in $\ci{D}$.

Let $T$ be a random vector in $[0,1)^d$.
Put 
\begin{gather*}
 (\ci{X}\!, \ci{\xi}) := \theta_{T} \theta_{-S} (X,\xi) \qquad \text{(change of origin).}
\end{gather*}
Our general assumption\, \eqref{convexhull}\, 
and the definition of the Voronoi cells easily imply
that \qquad($\bP$-a.e.) the number of elements in $D$ is finite, 
$|D| < \infty$, 
so  we can define
another measure $\ci{\bP}$ on~$(\gW, \cF)$ by
\begin{gather*}
  \bd\ci{\bP} := 
   \frac{\theta_{-S}\xi([0,1)^d)}{{|D|}} \,\bd{\bP}
   \qquad \text{(change of measure).}
\end{gather*}
Note that the distributions of 
$(X\!, \,\xi)$ and
 $(\ci{X}\!,\, \ci{\xi})$ are 
 $\gs$-finite under $\ci{\bP}$ if they are $\gs$-finite under $\bP$ (and vice versa).
 
 An informal explanation of the above construction of $(\ci{X}\!, \ci{\xi})$ and $\ci{\bP}$ 
 is given after the proof of the following theorem. 
 The proof is quite technical.
 
\begin{theorem}
\label{T:1}
Under $\bP$, let $(X,\xi)$  be stationary and 
\begin{equation*}
  \begin{aligned}
    {}&\txt{the conditional distribution of $T$
    given $(X, \xi)$ be $(\theta_{-S}\xi)(\, \cdot \given [0,1)^d)$.\,\,\,\,}
    \end{aligned}
\end{equation*} 
Then under $\ci{\bP}$\!, $(\ci{X}\!, \ci{\xi})$ is mass-stationary and
\begin{equation*}
  \begin{aligned}
      {}& \text{the conditional distribution of $T$ given $(\ci{X}\!, \ci{\xi})$ is uniform on $[0,1)^d$},\\
    {}& \text{the conditional distribution of  $S$ given $((\ci{X}\!, \ci{\xi}), T)$
    is uniform on $\ci{D}$}.\\
        \end{aligned}
\end{equation*}
Moreover, $(\ci{X}\!, \ci{\xi})$  under $\ci{\bP}$ is the Palm version
of $(X, \xi)$  under $\bP$.
\end{theorem}

\begin{proof}
We begin by proving that for all nonnegative measurable $f$, all Borel subsets $B$
of $[0,1)^d$ and all $i \in \bZ^d$,
\begin{align}\label{(4)}
\ci{\bE}[1_{\{S = i\}}1_{\{T \in B\}}f(\ci{X}\!, \ci{\xi})]
= \bE\Big{[}\frac{1_{\{i \in \ci{D}\}}}{{|D|}}\int_B f(\theta_t(X,\xi))\xi(dt)\Big{]}.
\end{align}
The definition of $\ci{\bP}$ and $(\ci{X},\ci{\xi})$ yields the first step in
\begin{align*}
\ci{\bE}&[1_{\{S = i\}}1_{\{T \in B\}}f(\ci{X}\!, \ci{\xi})] 
=
\bE\left[\frac{\theta_{-S}\xi([0,1)^d)}{{|D|}}1_{\{S = i\}}
1_{\{T \in B\}}f(\theta_T\theta_{-S}(X,\xi))\right]
\\ &= \bE\left[\frac{1_{\{S = i\}}}{{|D|}}\int_B f(\theta_t\theta_{-S}(X,\xi))
\theta_{-S}\xi(dt)\right]
\,\, \text{(conditional distribution of $T$)}
\\ &= \bE\left[\frac{1_{\{S = i\}}}{{|D|}}\int_B f(\theta_t\theta_{-i}(X,\xi))
\theta_{-i}\xi(dt)\right]\,\,\,\,\, \text{(use $S = i$)}.
\end{align*}
Now use $\theta_i(X,\xi)\overset{D}{=}(X,\xi)$ to replace
$\theta_{-i}(X,\xi)$
by $(X,\xi)$, $\theta_{-i}\xi(dt)$ by $\xi(dt)$, $S$~by~$S_i$ and $|D|$
by $|D_i|$ to obtain
\begin{align*}
\ci{\bE}&[1_{\{S = i\}}1_{\{T \in B\}}f(\ci{X}\!, \ci{\xi})] 
= \bE\left[\frac{1_{\{S_i = i\}}}{{|D_i|}}\int_B f(\theta_t(X,\xi))\xi(dt)\right].
\end{align*}
Note that
$\{S_i = i\} = \{S = 0, i \in D\}$ and that on this event $D_i = D=\ci{D}$. Thus
\begin{align*}
\ci{\bE}&[1_{\{S = i\}}1_{\{T \in B\}}f(\ci{X}\!, \ci{\xi})] 
= \bE\left[\frac{1_{\{S = 0,\, i \in  \ci{D}\}}}{{|D|}}\int_B f(\theta_t(X,\xi))\xi(dt)\right].
\end{align*}
Now \eqref{(4)} follows by noting that
when  $S \not= 0$ then $\xi(B)=0$ so the integral is $0$. 

Sum over  $i \in \bZ^d$ in \eqref{(4)} to obtain
\begin{gather}\label{(5)}
\ci{\bE}[1_{\{T \in B\}}f(\ci{X}\!, \ci{\xi})]
= \bE\Big{[}\int_B f(\theta_t(X,\xi))\xi(dt)\Big{]}.
\end{gather}
Use the stationarity to see that the measure defined by keeping $f$ fixed 
and letting $B$ on the right-hand side vary over the Borel subsets of $\bR^d$ 
is shift-invariant and thus of the form $\ci{\bE}[f(\ci{X}\!, \ci{\xi})]\gl$
where $\gl$ is the Lebesgue measure. 
This yields the Palm claim and, due to Theorem~1, the mass-stationarity claim.
This also yields
\begin{gather*}
\ci{\bE}[1_{\{T \in B\}}f(\ci{X}\!, \ci{\xi})]
= \ci{\bE}[\gl(B)f(\ci{X}\!, \ci{\xi})]\,\text{ for Borel subsets $B$ of  $[0,1)^d$.}
\end{gather*}
Since this holds for all nonnegative measurable $f$,
the conditional distribution of $T$ given $(\ci{X}\!, \ci{\xi})$ is uniform on $[0,1)^d$
under $\ci{\bP}$.

It only remains to establish that the conditional distribution of  
$S$ given $((\ci{X}\!, \ci{\xi}), T)$ is uniform on $\ci{D}$
under $\ci{\bP}$.
For that purpose, note that \eqref{(5)} implies that for all nonnegative measurable $g$
\begin{gather}\label{(6)}
\ci{\bE}[g(T, (\ci{X}\!, \ci{\xi}))]
= \bE\Big{[}\int_{[0,1)^d} g(t, \theta_t(X,\xi))\xi(dt)\Big{]}.
\end{gather}
For $t\in \bR^d$, let $D^{(-t)}$ be the analogue of $\ci{D}$ when $(X\!, \xi)$
is replaced by $\theta_{-t}(X, \xi)$. 
Fix an $i\in \bZ^d$ and define $g$ by 
\begin{align*}
g(t, (X, \xi)) &= \frac{1_{\{i \in D^{(-t)}\}}}{{|D^{(-t)}|}}1_{\{t \in B\}}f(X, \xi).
\end{align*}
Note that 
\begin{align*}
g(t, \theta_t(X, \xi)) = \frac{1_{\{i \in \ci{D}\}}}{{|D|}}1_{\{t \in B\}}f(\theta_t(X, \xi))
\end{align*}
and that, since $\theta_{-T}(\ci{X}\!, \ci{\xi}) = \theta_{-S}(X,\xi)$
and since $D^{(-S)} =  \ci{D}$, we also have 
\begin{align*}
g(T, (\ci{X}\!, \ci{\xi})) &= \frac{1_{\{i \in \ci{D}\}}}{{|D|}}1_{\{T \in B\}}f(\ci{X}\!, \ci{\xi}).
\end{align*}
Apply these two observations in \eqref{(6)} to obtain
\begin{gather*}
\ci{\bE}\left[\frac{1_{\{i \in \ci{D}\}}}{{|D|}}1_{\{T \in B\}}f(\ci{X}\!, \ci{\xi})\right]
= \bE \left[\frac{1_{\{i \in \ci{D}\}}}{{|D|}}\int_{B} f(\theta_t(X,\xi))\xi(dt)\right].
\end{gather*}
Compare this with \eqref{(4)} to get
\begin{align*}
\ci{\bE}[1_{\{S = i\}}1_{\{T \in B\}}f(\ci{X}\!, \ci{\xi})]
=\ci{\bE}\left[\frac{1_{\{i \in \ci{D}\}}}{{|D|}}1_{\{T \in B\}}f(\ci{X}\!, \ci{\xi})\right].
\end{align*}
This means that the conditional distribution of  
$S$ given $((\ci{X}\!, \ci{\xi}), T)$   is uniform on $\ci{D}$
under $\ci{\bP}$, and we are through.
\end{proof}

\noindent
The reason for the introduction of the point process $N$ is to enable a (shift-invariant) 
splitting of the mass of $\xi$ into finite clumps each having a reference point.
Now for each point $j$ of $N$, associate to each $i \in D_j$
the mass $\theta_j\xi([0,1)^d)/|D_j|$. Thus, all $i\in D_j$ 
have an equal share of the total mass of $\xi$ in the box $j+[0,1)^d$.
In order to guess at how the stationary $(X,\xi)$ 
might look when 
seen from a typical location in the mass of $\xi$, imagine  we could choose 
an $i\in \bZ^d$ according to this
redistribution of the mass of $\xi$.
Let $j$ be the point such that $i\in D_j$ and note that $j$ is
determined~by~$i$.
Choose $t \in [0,1)^d$
according to the probability measure $\theta_j \xi(\cdot \,| \,[0,1)^d )$
[note that, due to stationarity, $\theta_j \xi$ has no mass on the boundary
of the sets $[0,1)^d$ a.e.\! $\bP$].
Then $j + t$ would be placed in $\bR^d$ according 
 to the mass-distribution of~$\xi$.
Thus $(X, \xi)$ seen from 
 this typical location
 in the mass of $\xi$ should be mass-stationary.
 
Compare now the above informal argument
with the construction preceding Theorem~2.
Firstly, due to stationarity, what we see from $i$ [once $i$ has been chosen]
  is distributionally
the same as what we see from the origin $0$ after 
biasing  $\bP$ by $\theta_{-S}\xi([0,1)^d)/|D|$.
This fits with the definition of $\ci{\bP}$.
Secondly, 
$t$ 
is chosen according to $\theta_j \xi(\cdot | [0,1)^d )$
which only depends on $i$ through the point $j$.
This fits with the conditional distribution of $T$
given $(X,\xi)$  
being $\theta_{-S} \xi(\cdot | [0,1)^d )$.
Thirdly,
 the mass associated with $i$
is the same for all $i\in D_j$ and thus the vector $i-j$ 
is uniform in $D_j - j$.
This
fits with the conditional distribution of $S$ given 
$((\ci{X}\!, \ci{\xi}), T)$ being uniform on $\ci{D} = S + D$. 

It is not clear from the above discussion why the conditional distribution 
of $T$ given $(\ci{X}\!,\ci{\xi})$ should be uniform under $\ci{\bP}$.
However, according to the next theorem, 
this is exactly what is needed in order to reverse
the implication in Theorem 2.

\section{Construction of the stationary version}

\noindent 
In this 
  second construction
section, we construct the stationary Palm version 
when the mass-stationary version
is given.
Let $(\ci{X}, \ci{\xi}$) be mass-stationary under $\ci{\bP}$. 
Similarly as in the previous section we assume that
\begin{align}\label{convexhull2}
\operatorname{conv}(\operatorname{supp}\ci{\xi})=\bR^d\quad \text{$\ci{\bP}$-a.e.}
\end{align}
Since \eqref{convexhull2} is invariant under shifts of $\ci{\xi}$
it follows as at \eqref{convexhull} that \eqref{convexhull2} holds
if $(\ci{X}, \ci{\xi})$ is the Palm version
of a stationary pair with a finite distribution.

Let $T$ be a random vector in $[0,1)^d$ with $\ci{\bP}(T \notin (0,1)^d) = 0$.
Let $\ci{N}$ be the simple point process on $\bZ^d$ with a point at $i \in \bZ^d$ 
if and only if $\theta_{-T}\ci{\xi}(i + [0,1)^d) > 0$. 
Note that $\ci{N}(\{0\}) = 1$ since  $\ci{\xi}$ has $0$ in its support  ($\ci{\bP}$-a.e.).
Partition $\bZ^d$  into the Voronoi cells 
each containing 
exactly one point of $\ci{N}$.
Let $\ci{D}$ be the cell containing $0$. 

Let $S$ be a random vector taking values in $\ci{D}$ 
and put 
\begin{gather*}
 (X,\xi) :=  \theta_{S}\theta_{-T} (\ci{X}\!, \ci{\xi}) \qquad \text{(change of origin).}
\end{gather*}
Again our general assumption \eqref{convexhull2} 
and the definition of the Voronoi cells easily imply
that ($\ci{\bP}$-a.e.) the number of elements in $\ci{D}$ is finite,
$|\ci{D}| < \infty$.
Also since $\ci{\bP}(T \notin (0,1)^d) = 0$, we have $\ci{\bP}(\theta_{-T}\ci{\xi}([0,1)^d)=0)=0$.
Thus we can define another
 measure $\bP$ on $(\gW, \cF)$ by
\begin{gather*}
  \bd\bP := \frac{|\ci{D}|}{\theta_{-T}\ci{\xi}([0,1)^d)} \,\bd\ci{\bP}  \qquad \text{(change of measure).}
\end{gather*}
Note that 
the above construction of $(X, \xi)$ and  $\bP$ 
from $(\ci{X}, \ci{\xi})$ and  $\ci{\bP}$ is the reversal of the construction
in the previous section. 
We now reverse Theorem~2.

\begin{theorem}
\label{T:2}
Under $\ci{\bP}$\!, let $(\ci{X}\!, \ci{\xi})$  be mass-stationary and
\begin{equation*}
  \begin{aligned}
    {}& \text{the conditional distribution of $T$ given $(\ci{X}\!, \ci{\xi})$ be uniform on $[0,1)^d$},\\
    {}& \text{the conditional distribution of  $S$ given $((\ci{X}\!, \ci{\xi}), T)$
    be uniform on $\ci{D}$}.
        \end{aligned}
\end{equation*}
Then under $\bP$, $(X, \xi)$ is stationary and
\begin{equation*}
  \begin{aligned}
    {}&\txt{the conditional distribution of $T$
    given $(X, \xi)$ is $(\theta_{-S}\xi)(\,\cdot \given [0,1)^d)$.}\quad
    \end{aligned}
\end{equation*} 
Moreover, $(\ci{X}\!, \ci{\xi})$  under $\ci{\bP}$ is the Palm version
of $(X, \xi)$  under $\bP$.
\end{theorem}
\begin{proof}
Due to Theorem 1, there is a stationary
$(\hat{X}, \hat{\xi})$ defined on some measure space
$(\hat{\gW}, \hat{\cF}, \hat{\bP})$ 
such that
$(\ci{X}\!, \ci{\xi})$ under $\ci{\bP}$ is the Palm version of
$(\hat{X}, \hat{\xi})$. 
It is no restriction to let $(\hat{\gW}, \hat{\cF}, \hat{\bP})$
be large enough to support a $\hat{T}$ such that
\begin{align*}
\,\,\,\qquad\text{the conditional distribution 
of $\hat{T}$ given $(\hat{X}, \hat{\xi})$ is 
$(\theta_{-\hat{S}}\hat{\xi})(\,\cdot \given [0,1)^d)$}\qquad\qquad\,\,\,\,
\end{align*}
where $\hat{S}$ is obtained from $(\hat{X}, \hat{\xi})$
in the same way as $S$  from $(X, \xi)$.

Obtain $(\ci{\hat{X}}\!, \ci{\hat{\xi}})$ and $\ci{\hat{\bP}}$
from  $(\hat{X}, \hat{\xi}, \hat{T})$ and $\hat{\bP}$ in the same way 
as $(\ci{X}\!, \ci{\xi})$ and $\ci{\bP}$  in Section~3  
 is obtained from $(X, \xi, T)$ and $\bP$.
Then, due to Theorem~2, $(\ci{\hat{X}}\!, \ci{\hat{\xi}})$ under $\ci{\hat{\bP}}$
is the Palm version of the stationary $(\hat{X}, \hat{\xi})$.
But so is  $(\ci{X}\!, \ci{\xi})$ under $\ci{\bP}$. 
Therefore $(\ci{\hat{X}}, \ci{\hat{\xi}})$ under $\ci{\hat{\bP}}$
has the same distribution
as $(\ci{X}, \ci{\xi})$ under $\ci{\bP}$.
Also, due to Theorem~2 and our assumptions,
the conditional distribution of $(\hat{T}, \hat{S})$ given 
$(\ci{\hat{X}}\!, \ci{\hat{\xi}})$ under $\ci{\hat{\bP}}$ is
the same as that of  $(T, S)$ given $(\ci{X}, \ci{\xi})$ under $\ci{\bP}$.
Thus 
$$
\text{$(\ci{X}\!, \ci{\xi}\!,T, S)$ under $\ci{\bP}$
has the same distribution as $(\ci{\hat{X}}\!, \ci{\hat{\xi}}\!, \hat{T}, \hat{S})$ 
under $\ci{\hat{\bP}}$\!.}
$$
Now $(X, \xi, T)$ and $\bP$ are obtained in the same
way from $(\ci{X}, \ci{\xi},T, S)$ and $\ci{\bP}$
as $(\hat{X}, \hat{\xi}, \hat{T})$ and $\hat{\bP}$ 
from $(\ci{\hat{X}}, \ci{\hat{\xi}}, \hat{T}, \hat{S})$ and $\ci{\hat{\bP}}$.
Thus the distribution of $(X, \xi,T)$ under $\bP$ 
is the same as that of
$(\hat{X}, \hat{\xi}, \hat{T})$ under $\hat{\bP}$,
as desired. 
\end{proof}

Theorem 3 can be seen as an explicit version
of the {\em inversion formula} in~\cite{Mecke67};
see \cite[(2.7)]{LaTho09} for  a general version
of this formula.

\section{Constructions in the density case}

In this last construction section, let $\xi$ have a 
 density field $Z$, that is, 
let $Z = (Z_s)_{s\in\bR^d}$ be a shift-measurable
random field taking values in $[0,\infty)$ and such that
\begin{gather*}
\xi(ds) = Z_s ds.
\end{gather*} 
We shall now show that in this case
there is no need for a change of origin
 in order to go from stationarity to mass-stationarity, 
only a change of measure is needed.
\begin{theorem}
Let  
$\xi$ have a density  field $Z$.  
If $(X, Z)$ is stationary under a measure $\bP$
then $((X, Z), \xi)$ is mass-stationary under the measure $\ci{\bP}$
defined by 
\begin{gather}\label{dd}
    \bd\ci{\bP} = Z_0\bd\bP
\end{gather}
and $((X, Z), \xi)$  under $\ci{\bP}$ is the Palm version of 
$((X, Z), \xi)$  under $\bP$.
\end{theorem}
\begin{proof}
Let $B$ be a Borel subset of $\bR^d$ such that 
 $0 < \gl(B) < \infty$ and let $f$ be a nonnegative measurable function. Then
\begin{align*}  \label{a}
 \ci{\bE}[&f(((X, Z), \xi))]  
 = \bE[f(((X, Z), \xi)) Z_0] \quad\,\,\,\,  
  \text{(by definition of \ci{\bP})} \\
 & = \frac{1}{\lambda(B)}\int_B \! \bE[f(\gq_s ((X, Z), \xi)) Z_s]\mrm{d}s  
 \qquad\,  \text{(by stationarity)}\\
  & = \frac{1}{\lambda(B)}\bE\!\left[\int_B f(\gq_s ((X, Z), \xi)) 
  \xi(\mrm{d}s)\right] 
  \quad\,\, \text{(by Fubini).}
\end{align*}
Thus  $((X, Z), \xi)$ under $\ci{\bP}$ is the Palm version of  $((X, Z), \xi)$
under $\bP$. And mass-stationarity follows from Theorem~1.
\end{proof}

In order to reverse this theorem, -- go from mass-stationarity to stationarity 
without a change of origin, -- we shall assume that 
the density  field is strictly positive. 

\begin{theorem}
Let  
$\xi$ have a density  field 
$Z$ 
which is strictly positive everywhere.
If $((X, Z), \xi)$ is mass-stationary under a measure $\ci{\bP}$
then $(X, Z)$ is stationary under the measure $\bP$
defined by 
\begin{gather}
\bd\bP = \frac{1}{Z_0}\bd\ci{\bP}
\end{gather}
and $((X, Z), \xi)$  under $\ci{\bP}$ is the Palm version of 
$((X, Z), \xi)$  under $\bP$.
\end{theorem}

\begin{proof}
Due to Theorem~1, $((X, Z), \xi)$ 
under $\ci{\bP}$ 
is the Palm version of some stationary $((\hat{X}, \hat{Z}), \hat{\xi})$ 
defined on some $(\hat{\gW}, \hat{\cF}, \hat{\bP})$.
Due to Theorem~3,  $((\hat{X}, \hat{Z}), \hat{\xi})$ 
can in fact be obtained by shifting the paths of $((X, Z), \xi)$ itself
and thus $\hat{Z}$ will be a density field of $\hat{\xi}$.
Due to Theorem~4, $((\hat{X}, \hat{Z}), \hat{\xi})$ under the 
changed measure 
$\hat{Z}_0\bd\hat{\bP}$ is also
the Palm version
and thus has the same distribution as $((X, Z), \xi)$ under $\ci{\bP}$.
Now $\bd\hat{\bP}$ is recovered from $\hat{Z}_0\bd\hat{\bP}$ by dividing by 
$\hat{Z}_0$ just like 
$\bd\bP$ is obtained from $\bd\ci{\bP}$ by dividing by 
$Z_0$. This yields that $((X, Z), \xi)$ 
under $\bP$ has the same distribution as 
$((\hat{X}, \hat{Z}), \hat{\xi})$ under $\hat{\bP}$, as desired.
\end{proof}

\section{Characterisation in the positive density case}

\noindent 
In this first characterisation section, we let $\xi$ have a strictly positive
density field $Z$ 
and establish a natural shift characterization of mass-stationarity.

Let $\gp$ be a measurable map taking $Z$ to 
a location $\gp(Z)$ in $\bR^d$. 
Define the induced {\em allocation rule} $\gt = \gt_{\gp}^{Z}$ by
\begin{gather*}
  \gt
  (s) \defas s + \gp(\gq_s Z), \bil s \in \bR^d.
\end{gather*}
Call $\gp$ 
a {\em preserving shift} if 
for each fixed value of $\xi$ the allocation rule $\gt$
preserves~$\xi$, 
\begin{gather*}
  \xi(\gt \in \cdot)   = \xi, \quad 
  \text{that is, \quad 
   $\xi(\{s \in \bR^d :
   \gt(s) \in B\}) 
   = \xi(B)$ for Borel $B \subseteq \bR^d$}.
\end{gather*}
Say that $((X, Z),  \xi)$ is \emph{distributionally invariant under preserving shifts}
(under a measure $\ci{\bP}$) if for all preserving $\gp$
\begin{gather*}
  \gq_{\gp(Z)}((X, Z), \xi)
   \overset{D}{=} ((X, Z), \xi) \bil (\text{under } \ci{\bP}). 
\end{gather*}

In the case when $\xi$ is a simple point process
it is proved in \cite{He:La:05} 
that distributional invariance under 
preserving shifts is a characterization of Palm versions of stationary pairs,
and thus (due to 
Theorem~1) it is also a characterization of mass-stationary pairs.
Also  in \cite{LaMoeTho12} it is proved
that the same is true for diffuse random measures when $d=1$.
We shall now prove that this is still true 
in a positive density case for any $d\geq 1$.
This provides a partial solution to 
Problem 7.3 in \cite{LaTho09}.
\begin{theorem}
Let $\xi$ have a density field
$Z$  such that $Z_0 > 0$ everywhere
and such that $Z$ is locally integrable along all lines
and has infinite integral along all half-lines. 
Let $X$ and $Z$ be defined on $(\Omega, \cF , \ci{\bP})$.
Then $((X, Z), \xi)$ is mass-stationary 
if and only if $((X, Z), \xi)$ is distributionally invariant under preserving shifts.
\end{theorem}
\begin{proof} 
The only-if-direction follows from Theorem 7.2 in \cite{LaTho09}.
In order to establish the if-direction, assume
that $((X, Z),  \xi)$ is distributionally invariant under preserving shifts.
Note that if $\bP$ is the measure 
defined at (9)
then we recover our $\ci{\bP}$ as the measure defined 
at~\eqref{dd}. 
Thus, due to Theorem~4, 
$((X, Z), \xi)$ is mass-stationary (under $\ci{\bP}$) if we can show that
$((X, Z), \xi)$ is stationary under $\bP$,
that is, if we can show that  for all $t \in \bR^d$ and all nonnegative measurable $f$,
\begin{gather}  \label{B}
  \ci{\bE}[f(\gq_t ((X, Z), \xi))/Z_0] = \ci{\bE}[f(((X, Z), \xi))/Z_0].  
\end{gather}
For that purpose, fix $t \neq 0$ and write $t = au$ where $a > 0$ and $u$ is a vector of length 1. For $r > 0$, define $s_r(Z)$ 
by
\begin{gather*}
  \int_0^{s_r(Z)}Z_{xu}\mrm{d}x = r.
\end{gather*}
and define a preserving [see \cite{LaMoeTho12}, Theorem 3.1] shift $\pi_r$ by
\begin{gather*}
  \gp_r(Z) = s_r(Z)u.
\end{gather*}
Take $h>0$ and use the fact that $((X, Z),  \xi)$ 
is distributionally
invariant under $\gp_r$
to obtain the first step in
\begin{align*}
  \ci{\bE}[&f(((X, Z), \xi))/Z_0]  = \frac{1}{h}\int_0^h\ci{\bE}[f(\gq_{s_r(Z)u}((X, Z), \xi))/Z_{s_r(Z)u}]\mrm{d}r\\
  & = \frac{1}{h}\ci{\bE}\left[\int_0^h f(\gq_{s_r(Z)u}((X, Z), \xi))/Z_{s_r(Z)u}\mrm{d}r\right] \quad \text{(by Fubini).}
\end{align*}
Now apply variable substitution, $s = s_r(Z)$ and $\mrm{d}r = Z_{su}\mrm{d}s$, to obtain
\begin{align*}
  \ci{\bE}[&f(((X, Z), \xi))/Z_0]  = 
   \frac{1}{h}\ci{\bE}\left[\int_0^{s_h(Z)}f(\gq_{su}((X, Z), \xi))\mrm{d}s\right].
\end{align*}
Apply this with $f$ replaced by $f \circ \gq_t$ (and remember $t = au$) to obtain
\begin{gather*}
  \ci{\bE}[f(\gq_t ((X, Z), \xi))/Z_0] = \frac{1}{h}\ci{\bE}\left[\int_a^{a+s_h(Z)}f(\gq_{su}((X, Z), \xi))\mrm{d}s\right].
\end{gather*} 
Thus for $0 \leq f \leq 1$,
\begin{align*}
  |\ci{\bE}[&f(\gq_t ((X, Z), \xi))/Z_0]  - \ci{\bE}[f(((X, Z), \xi))/Z_0]| \\
  & \leq \frac{1}{h}\ci{\bE}\left[\int|1_{[a, a+s_h(Z)]}(s) - 1_{[0, s_h(Z)]}(s)|f(\gq_s ((X, Z), \xi))\mrm{d}s\right] \\
  & \leq 2\,\frac{a}{h} \to 0, \pbil h \to \infty.
\end{align*}
Thus \eqref{B} holds, as desired.
\end{proof}

According to Theorem 3.1 in \cite{LaMoeTho12}, the shift characterization of
mass-stationarity in Theorem~6 above works
when $d=1$ and $\ci{\xi}$ is only diffuse.
Thus, when $d=1$, the  
 background randomization
 in the next section is not needed.

\section{Characterisation in the diffuse case}
\noindent
In this second characterisation section, let $\ci{\xi}$ be diffuse, that is, 
let it have no atoms,
$$\ci{\xi}(\{t\}) = 0, \quad
t \in \bR^d.$$
We shall show that the shift characterization of
mass-stationarity in Theorem~6 works in this case 
if we apply the following background randomization 
introduced for point processes in 
\cite{Thor00}.

Let $\ci{Y}$
be a random element in a space on which the additive group $\bR^d$
acts measurably. For instance,  $\ci{Y}$ could be a 
random field
$\ci{Y} = (\ci{Y}_s)_{s \in \bR^d}$.
Call $\ci{Y}$ a {\em stationary independent background} if
$\ci{Y}$ is stationary and independent of $(\ci{X}, \ci{\xi})$ and
possibly obtained by extending the underlying space
$(\Omega, \cF, \ci{\bP})$ supporting $(\ci{X}, \ci{\xi})$.
Let $\gp$ be a measurable map taking 
$(\ci{Y}\!, \ci{\xi})$ to a location $\gp(\ci{Y}\!, \ci{\xi})$ in $\bR^d$.
Define the induced {\em allocation rule} 
$\gt = \gt_{\gp}^{(\ci{Y}\!,\, \ci{\xi})}$ by
\begin{gather*}
  \gt
  (s) \defas s + \gp(\gq_s (\ci{Y}\!, \ci{\xi})), \bil s \in \bR^d.
\end{gather*}
Call $\gp$ 
a {\em preserving shift} if 
for each fixed value of $(\ci{Y}\!, \ci{\xi})$ the allocation rule $\gt$
preserves~$\ci{\xi}$, 
\begin{gather*}
  \ci{\xi}(\gt \in \cdot)   = \ci{\xi}, \quad 
\end{gather*}
that is,
\begin{gather*}
  \text{
   $\ci{\xi}(\{s \in \bR^d :
   \gt(s) \in B\}) 
   = \ci{\xi}(B)$\quad for Borel $B \subseteq \bR^d$}.
\end{gather*}
Say that $(\ci{X}\!, \ci{\xi})$ is \emph{distributionally invariant under preserving shifts against any stationary independent background} 
if for all stationary independent backgrounds $\ci{Y}\!$
and preserving shifts $\gp$
\begin{gather*}
  \gq_{\gp(\ci{Y}\!, \ci{\xi})}(\ci{Y}\!, \ci{X}, \ci{\xi})
   \overset{D}{=} (\ci{Y}\!, \ci{X}, \ci{\xi}) \bil (\text{under } \ci{\bP}). 
\end{gather*}
Here is a key example of such $\ci{Y}$ and $\gp$.
\vspace{.2cm}

\noindent
{\bf Example 1.} Fix $n \in \bN$ and let $\gj$ be a Borel equivalence between $[0,n)^d$ and $[0,1)$. We can for instance choose $\gf$ as follows. 
For $s = (s_1,\dots,s_d) \in  [0,n)^d$, write
$s_k/n$  in binary form as $.a_{k1}0a_{k2}0\dots$ where $a_{k1},
a_{k2}, \dots$ are finite strings of the number $1$ possibly of
length zero. Put $\gf(s) = .a_{11}0\dots a_{d1}0a_{12}0\dots a_{d2}0\dots$
This mapping is measurable and has a measurable inverse. 

Let $\gm$ be a diffuse probability measure on $[0,n)^d$,
let $P_{\gm}$ be the distribution of the 
$[0,1)$ valued function $\gf$ under $\mu$,
 and let $F_{\gm}$ be its distribution function 
defined on $[0,1)$ by
\begin{gather*}
  F_{\gm}(x) = \gm(\gj \leq x) = P_{\gm}((-\infty,x]), \bil x \in [0,1).
\end{gather*}
Since $\gm$ has no atom and $\gj$ is a bijection, $F_{\gm}$ is continuous and
\begin{gather*}
  \text{under $P_{\gm}$, \bil$F_{\gm}$ is uniform on $[0,1)$.}
\end{gather*}
Thus, for each $r \in [0,1),$
\begin{gather*}
  \text{under $P_{\gm}$, \bil$(F_{\gm} + r \; \bmod 1)$ is uniform on $[0,1),$}
\end{gather*}
and, with $F_{\gm}^{-1}$ the (left- or right-continuous) generalized inverse of $F_{\gm},$
\begin{gather*}\label{9a}
  \text{under $P_{\gm}$,  \bil$F_{\gm}^{-1}(F_{\gm} + r \; \bmod 1)$ has distribution $P_{\gm}$.}
\end{gather*}
Since, by definition, $\gf$ has distribution $P_{\gm}$ under $\gm$, and 
since $\gf$ has a measurable inverse $\gj^{-1} $, this 
implies that

\vspace*{-8mm}

\begin{gather}\label{psi}
  \text{under $\gm,$ \bil $\psi_r^{\gm}$ has distribution $\gm$,}
\end{gather}

\vspace*{1mm}

\noindent
where $\psi_{r}^{\gm}$ is the function from $[0,1)^d$ to $[0,1)^d$ defined by

\vspace*{-4mm}

\begin{gather*}
  \psi_{r}^{\gm} = \gj^{-1}(F_{\gm}^{-1}(F_{\gm}(\gj) + r \; \bmod 1));
\end{gather*}

\vspace*{1mm}

\noindent
so $\psi_{r}^{\gm}$ is the combined map in 
the following diagram where $\gl_{[0,1)}$ denotes Lebesgue measure on $[0,1)$: 

\vspace*{-11mm}

\begin{gather*}
  \xymatrix{
    \overset{\gm}{[0,n)^d} \ar[r]^{\gj} &
    \overset{P_{\gm}}{[0,1)} \ar[r]^{F_{\gm}} &
    \overset{\gl_{[0,1)}}{[0,1)} \ar[d]^{+ \; r \; \bmod 1} \\
    \overset{\gm}{[0,n)^d} &
    \overset{P_{\gm}}{[0,1)} \ar[l]_{\gj^{-1}} &
    \overset{\gl_{[0,1)}}{[0,1)} \ar[l]_{F_{\gm}^{-1}}
  }  
\end{gather*}
Let $\ci{Y} = (\ci{Y}_s)_{s \in \bR^d}$ be the stationary independent background
defined as follows. Let $\ci{Y}_0$ be 
independent of $(\ci{X}, \ci{\xi})$ 
and uniform on $[0,n)^d$ and let $\ci{Y}_s$ be the vector from the 
lexicographically lowest corner of the $n\bZ^d - \ci{Y}_0$ 
box, in which $s \in \bR^d$ lies, to $s$. For 
$r \in [0,1),$ define $\gp_r$ by
\begin{gather}  \label{1}
  \gp_r(\ci{Y}, \ci{\xi}) \defas \psi_{r}^{\gm}(\ci{Y}_0) - \ci{Y}_0 
  \bil \text{where} \bil
  \gm = \gq_{-\ci{Y}_0}\ci{\xi}(\cdot \given [0,n)^d).
\end{gather}
The allocation rule $\tau_r$ induced by $\gp_r$ is
\begin{gather*}
  \gt_r
  (s) \defas 
s - \ci{Y}_s + \psi_r^{\gq_{s-\ci{Y}_s}\ci{\xi}(\cdot \given [0,n)^d)}
(s - \ci{Y}_s), \bil s \in \bR^d.
\end{gather*}
Since $s - \ci{Y}_s$  is the lowest corner of a box
and since (according to \eqref{psi}) $\psi_r^{\gm}$ preserves a diffuse measure $\gm$ on
the box $[0,n)^d$,
it follows that $\gt_r$
preserves $\ci{\xi}$ within each box of $n\bZ^d - \ci{Y}_0$.
 Thus $\gp_r$ is preserving. 
 
 Note that  in the above notation we have suppressed the parameter
 $n \in \bN$ 
introduced at the beginning of the example. 
To make the dependence on $n$ explicit
write $\pi_r^{(n)}$ for $\gp_r$ and $Y^{(n)}$ for $\ci{Y}$. 
 \qed
\vspace{0.4cm}

The next theorem gives a randomized-background characterisation
of mass-stationarity linking it to 
 the definition of point-stationarity
given in Chapter 9 of \cite{Thor00}.
It provides a partial solution to Problem~7.6 in \cite{LaTho09}.

\begin{theorem}
\label{T:11}
Let $\ci{\xi}$ be a diffuse random measure on $\bR^d$. Then
\emph{(}\ci{X}\!, \ci{\xi}\emph{)} is mass-stationary
if and only if
\emph{(}\ci{X}\!, \ci{\xi}\emph{)}  is distributionally invariant under preserving shifts 
against any independent stationary background.
Moreover, \emph{(}\ci{X}\!, \ci{\xi}\emph{)} is mass-stationary
if and only if 
\begin{gather}\label{Ex}
  \gq_{\gp_r^{(n)}(Y^{(n)}\!,\, \ci{\xi})}(Y^{(n)}\!, \ci{X}\!, \ci{\xi})
   \overset{D}{=} (Y^{(n)}\!, \ci{X}\!, \ci{\xi}) 
\end{gather}
for all the stationary independent backgrounds $Y^{(n)}$ and 
preserving shifts $\pi_r^{(n)}$,  $n \in \bN$, $r \in [0,1)$, defined in Example 1.
\end{theorem}
\begin{proof}
The only-if-results
follow from Theorem 7.2 in \cite{LaTho09}.
Since the latter if-result 
is stronger
than the first,
it only remains to show that \eqref{Ex} implies mass-stationarity.
For that purpose, assume that \eqref{Ex} holds 
and write $\gp_r$ for  $\pi_r^{(n)}$ 
and $\ci{Y}$ for $Y^{(n)}$.
Let $\ci{R}$ be
uniform on $[0, 1)$ and  independent of $(\ci{Y}, \ci{X}, \ci{\xi})$.

Now 
consider the lines 
from  \eqref{psi} to \eqref{1} and put 
$$
\ci{T} := \pi_{\ci{R}}(\ci{Y}, \ci{\xi}) \defas \psi_{\ci{R}}^{\gm}(\ci{Y}_0) - \ci{Y}_0 
  \bil \text{where} \bil
  \gm = \gq_{-\ci{Y}_0}\ci{\xi}(\cdot \given [0,n)^d).
$$
Note that 
given $(\ci{Y}\!, \ci{X}, \ci{\xi})$ the conditional distribution of
$(F_{\gm}(\gj(\ci{Y}_0)) + \ci{R}\; \bmod 1)$
is uniform on $[0, 1)$.
This implies that given $(\ci{Y}\!, \ci{X}, \ci{\xi})$ the conditional distribution of 
$F_{\gm}^{-1}(F_{\gm}(\gj(\ci{Y}_0)) + \ci{R}\; \bmod 1 )$
is~$P_\gm$.
This implies further that given $(\ci{Y}\!, \ci{X}, \ci{\xi})$ the conditional distribution of 
$\psi_{\ci{R}}^\mu(\ci{Y}\!, \ci{X}, \ci{\xi})$
is  $\gm = \gq_{-\ci{Y}_0}\ci{\xi}(\cdot \given [0,n)^d)$.
Therefore  
$$\text{given  $(\ci{Y}, \ci{X}, \ci{\xi})$
the conditional distribution of 
$\ci{T}$
is $\ci{\xi}(\cdot \given [0,n)^d -\ci{Y}_0)$.}$$
Since the conditional distribution of $\ci{Y}_0$ given
 $(\ci{X}, \ci{\xi})$
 is uniform on $[0,n)^d$ this yields 
\begin{gather}  \label{A1}
(\ci{X}, \ci{\xi}, \ci{Y}_0,\ci{T}) \overset{D}{=}(\ci{X}, \ci{\xi},U_{[0,n)^d}, V_{[0,n)^d})
\end{gather}
where $U_{[0,n)^d}$ and $V_{[0,n)^d}$ are from the definition of mass-stationarity 
at \eqref{(1)}.

Since $\ci{R}$ is independent of $(\ci{Y}, \ci{X}, \ci{\xi})$
we obtain from \eqref{Ex} that
the conditional distribution of
$
\theta_{\ci{T}}(\ci{Y}, \ci{X}, \ci{\xi})$ given $\ci{R}$
is the distribution of 
$(\ci{Y}, \ci{X}, \ci{\xi}).
$
Thus 
\begin{gather*}
  \gq_{\ci{T}}(\ci{Y}\!, \ci{X}\!, \ci{\xi})
   \overset{D}{=} (\ci{Y}\!, \ci{X}\!, \ci{\xi})  
\end{gather*}
which implies
\begin{gather}  \label{A3}
(\theta_{\ci{T}}(\ci{X}, \ci{\xi}), \ci{Y}_{\ci{T}}) \overset{D}{=} ((\ci{X}, \ci{\xi}), \ci{Y}_0).
\end{gather}
From  \eqref{A1} and $ \ci{T} + \ci{Y}_0 = \ci{Y}_{\ci{T}}$ we obtain
$$(\theta_{V_{[0,n)^d}}(\ci{X}, \ci{\xi}),V_{[0,n)^d}+U_{[0,n)^d}) 
\overset{D}{=}  (\theta_{\ci{T}}(\ci{X}, \ci{\xi}), \ci{Y}_{\ci{T}}) $$
which together with  \eqref{A3} and \eqref{A1} yields
\begin{align*} 
(\theta_{V_{[0,n)^d}}(\ci{X}, \ci{\xi}),V_{[0,n)^d}+U_{[0,n)^d}) \overset{D}{=}
((\ci{X}, \ci{\xi}), U_{[0,n)^d}).
\end{align*}
Now mass-stationarity follows from Lemma~2 below.
\end{proof}

The following lemma is 
needed in the proof of Lemma~2 below,
but it is 
quite interesting on its own.

\vspace{.2cm}
\noindent
{\bf Lemma 1.} \emph{
Suppose \eqref{(1)} holds for a bounded Borel set $C$ with $\lambda(C) > 0$ and 
$\lambda(\partial C) = 0$. Then
\vspace{-4mm}
\begin{gather}\label{(13)}
(\theta_{V_C}(\ci{X},\ci{\xi}),V_C + U_C, U_C)\overset{D}{=}((\ci{X},\ci{\xi}), U_C, V_C + U_C).
\end{gather}
}

\vspace{-8mm}
\begin{proof}
Let $f$ be the bijection defined by 

\vspace{-3mm}

$$f((\ci{X},\ci{\xi}),U_C) = 
(\theta_{-U_C}(\ci{X},\ci{\xi}), U_C)$$ 

\vspace{-3mm}

\noindent
and note that 

\vspace{-5mm}

$$f(\theta_{V_C}(\ci{X},\ci{\xi}),V_C + U_C) = 
(\theta_{-U_C}(\ci{X},\ci{\xi}), V_C + U_C).$$
Thus  \eqref{(1)} is equivalent to
\vspace{-3mm}
\begin{gather}\label{B1}
(\theta_{-U_C}(\ci{X},\ci{\xi}),V_C + U_C)\overset{D}{=}
(\theta_{-U_C}(\ci{X},\ci{\xi}), U_C).
\end{gather}

\noindent
Similarly, let $g$ be the bijection defined by 
$$g((\ci{X},\ci{\xi}),U_C, V_C + U_C) = 
(\theta_{-U_C}(\ci{X},\ci{\xi}), U_C, V_C + U_C)$$ 

\vspace{-3mm}

\noindent
and note that 

\vspace{-5mm}

$$g(\theta_{V_C}(\ci{X},\ci{\xi}),V_C + U_C, U_C) = 
(\theta_{-U_C}(\ci{X},\ci{\xi}), V_C + U_C, U_C).$$
Thus \eqref{(13)} is equivalent to
\vspace{-5mm}

\begin{gather}\label{C}
(\theta_{-U_C}(\ci{X},\ci{\xi}),V_C + U_C, U_C)\overset{D}{=}
(\theta_{-U_C}(\ci{X},\ci{\xi}), U_C, V_C + U_C).
\end{gather}
Note that the conditional distribution of $V_C+U_C$ given 
$(\theta_{-U_C}(\ci{X}\!,\ci{\xi}), U_C)$
is $(\theta_{-U_C}\ci{\xi})(\cdot\!\given\! C)$.
Thus $V_C+U_C$\, 
and $U_C$ are conditionally independent given 
$\theta_{-U_C}(\ci{X},\ci{\xi})$, see e.g.\! Proposition 66 in \cite{Kallenberg}.
Since also, due to \eqref{B1}, 
the 
conditional distribution of $V_C+U_C$ given 
$\theta_{-U_C}(\ci{X},\ci{\xi})$ is the same as that of $U_C$ given 
$\theta_{-U_C}(\ci{X},\ci{\xi})$,
we obtain
 \eqref{C}. 
\end{proof}

The next lemma reduces the class of sets $C$ 
needed to define mass-stationarity.
It was used in the proof of Theorem~\ref{T:11}
and will be used in the proof of Proposition 1 below.

\vspace{.2cm}
\noindent
{\bf Lemma 2.} \emph{
Suppose \eqref{(1)} holds for all \,$C \!= [0, n)^d$, $n \in \bN$. Then 
\emph{(}\ci{X}\!, \ci{\xi}\emph{)} is mass-stationary.}

\vspace{.2cm}
\begin{proof}
Apply Lemma 1 with $C \!= [0, n)^d$
to obtain
\begin{align*}
(\theta_{V_{[0, n)^d}}(\ci{X},\ci{\xi}),V_{[0, n)^d} + U_{[0, n)^d}, U_{[0, n)^d})
\overset{D}{=}((\ci{X},\ci{\xi}), U_{[0, n)^d}, V_{[0, n)^d} + U_{[0, n)^d}).
\end{align*}
To establish \eqref{(1)} for an arbitrary bounded 
$C$ with $\lambda(C) >0$ and 
$\lambda(\partial C) = 0$,
note that it is no restriction to assume
that there is~an~$n$ such that $C \subseteq  [0, n)^d$.
Then \eqref{(13)}  follows by conditioning on both sides in the last display by 
the event $\{U_{[0, n)^d}\! \in C, V_{[0, n)^d} + U_{[0, n)^d}\!\in{}C\}$.
This yields \eqref{(1)}.
\end{proof}

\section{Characterisation in the general $\bR^d$ case}

In this last characterisation section, we shall finally allow  $\ci{\xi}$ to be general.
The background 
randomization, applied to diffuse $\ci{\xi}$ in the previous section, 
does not work in the general
case;\!
for\! a\! counterexample\! add \!a\! stationary \!independent \!background \!in\!
Example\!~7.1\!~of\!~\cite{LaTho09}.
However, this can be mended by a 
simple extension of the background idea.

Let $\lambda_1$ be the Lebesgue measure on $\bR$.
Recall that for each $t \in \bR^d$, the map $\theta_t$ 
takes $x \in E$ 
to $\theta_t x \in E$.
Extend this class of maps
from $\bR^d$ to $\bR^{d+1}$ as follows:
for $t = (t_1, \dots, t_{d+1}) \in \bR^{d+1}$ put 
$\theta_t x = \theta_{(t_1,\dots,t_d)}x$.

\vspace{.2cm}
\noindent
{\bf Proposition 1.} \emph{
The pair $(\ci{X}\!, \,\ci{\xi})$ is mass-stationary
if and only if $(\ci{X}\!, \,\ci{\xi}\!\otimes\lambda_1)$
is mass-stationary.
}
\vspace{.2cm}
\begin{proof}
Let $U_{[0,n)^d}$ and $V_{[0,n)^d}$ be as
in the definition of mass-stationarity~at~\eqref{(1)}.
Let the conditional distribution of $U_{[0,n)^{d+1}}$ given 
$$(\ci{X}\!,\,\ci{\xi}\!\otimes\lambda_1)$$
be uniform on $[0,n)^{d+1}$, and let the conditional distribution of $V_{[0,n)^{d+1}}$ 
given $$((\ci{X}\!,\,\ci{\xi}\!\otimes\lambda_1), U_{[0,n)^{d+1}})$$ be
$(\ci{\xi}\!\otimes\lambda_1)(\cdot \given [0,n)^{d+1} - U_{[0, n)^{d+1}})$.
Then it is easily seen that
\begin{gather*}
(\theta_{V_{[0,n)^{d}}}(\ci{X}\!,\,\ci{\xi}),
V_{[0,n)^{d}} + U_{[0,n)^{d}})\overset{D}{=}((\ci{X}\!,\,\ci{\xi}),U_{[0,n)^{d}})
\end{gather*}
is equivalent to
\begin{gather*}
(\theta_{V_{[0,n)^{d+1}}}(\ci{X}\!,\,\ci{\xi}\!\otimes \lambda_1),
V_{[0,n)^{d+1}} + U_{[0,n)^{d+1}})\overset{D}{=}
((\ci{X}\!,\,\ci{\xi}\!\otimes \lambda_1),U_{[0,n)^{d+1}}).
\end{gather*}
Thus, due to Lemma 2 above, 
$(\ci{X},\ci{\xi})$ is mass-stationary
if and  only if 
$(\ci{X}\!,\,\ci{\xi}\otimes \lambda_1)$ is mass-stationary.
\end{proof}

The following simple extension of the background randomisation idea 
from the previous section
works
without restrictions on $\ci{\xi}$.

\begin{theorem}
The pair $\emph{(}\ci{X}\!,\, \ci{\xi}\emph{)}$ is mass-stationary
if and only if $\emph{(}\ci{X}\!, \,\ci{\xi}\otimes\lambda_1\emph{)}$ 
is
distribution-ally invariant under preserving shifts 
against any independent stationary background.
\end{theorem}
\begin{proof}
Note that 
$\ci{\xi}\otimes\lambda_1$
is a diffuse random measure on $\bR^{d+1}$.
Thus the theorem follows from 
Theorem~7 and Proposition~1.
\end{proof}

\section{Final remark
on mass-stationarity}

In this final section we interpret Theorems 7 and 8
in terms of the transport formulae derived  in \cite{LaTho09}
and further developed in \cite{Last10a,GentnerLast11,Kallenberg11}.

Let $(X^\circ,\xi^\circ)$ be as in Section 8 and consider a stationary 
independent background $Y^\circ$ as in Section 7, defined on
$(\Omega,\mathcal{F},\bP^\circ)$. Consider an allocation rule
$\tau(Y^\circ,X^\circ,\xi^\circ,s)\equiv\tau(s):=s+\pi(\theta_s(Y^\circ,X^\circ,\xi^\circ))$,
$s\in\bR^d$, where $\pi$ is a measurable mapping with values in
$\bR^d$. Let $C$ be a measurable subset of the path space
of $Y^\circ$ such that $0<\ci{\bP}(Y^\circ\in C)<\infty$
and define  a kernel $K_{\pi,C}$ by
\begin{align}\label{kernel1}
K_{\pi,C}(X^\circ\!,\xi^\circ\!,s,B):=\frac{1}{\ci{\bP}(Y^\circ\!\!\in C)}
\int \!1\{\tau(y,X^\circ\!,\xi^\circ\!,s)\!\in\! B,\theta_{\tau(y,X^\circ\!,\xi^\circ\!,s)}y\!\in\! C\}
\ci{\bP}(Y^\circ\!\!\in\! dy)
\end{align}
where $B\subset\bR^d$ is a Borel set and $s\in\bR^d$. Using stationarity of
$Y^\circ$
it is not hard to check that $K_{\pi,C}$ is {\em invariant} in the sense that
\begin{align}\label{kernel2}
K_{\pi,C}(\theta_s(X^\circ,\xi^\circ),0,B-s)=K_{\pi,C}(X^\circ,\xi^\circ,s,B).
\end{align}
Hence $K_{\pi,C}$ is an invariant {\em weighted transport kernel} in the sense
of \cite{LaTho09}. It is also easy to see that $\pi$ is preserving
(i.e.\ $\xi^\circ(\tau\in\cdot)=\xi^\circ$) if and only if
$K_{\pi,C}$ is preserving for all $C$, that is
\begin{align}\label{kernel3}
\int K_{\pi,C}(X^\circ,\xi^\circ,s,B)\xi^\circ(ds)=\xi^\circ(B).
\end{align}
Note that $K_{\pi,C}$ depends only on the original pair $(X^\circ,\xi^\circ)$,
but not on the background $Y^\circ$. If $(X^\circ,\xi^\circ)$ is mass-stationary and
\eqref{kernel2} holds, then \cite[Theorem 4.1]{LaTho09} implies that
\begin{align}\label{kernel4}
\bE^\circ\int 1\{\theta_t(X^\circ,\xi^\circ)\in \cdot\}K_{\pi,C}(X^\circ,\xi^\circ,0,dt)
=\bP^\circ((X^\circ,\xi^\circ)\in \cdot),
\end{align}
or, equivalently,
\begin{align*}
\theta_{\pi(Y^\circ\!,\,X^\circ\!,\,\xi^\circ)}(Y^\circ,X^\circ,\xi^\circ)\overset{D}{=}(Y^\circ,X^\circ,\xi^\circ)\quad
\text{(under $\bP^\circ$)}.
\end{align*}

Theorem 7 says for a diffuse $\xi^\circ$ that mass-stationarity of
$(X^\circ,\xi^\circ)$ is equivalent to the distributional invariance
\eqref{kernel4} for all kernels $K_{\pi,C}$ of the form \eqref{kernel1}
(for preserving $\pi$ not depending on $X^\circ$). It is interesting to note that these
kernels are not Markovian, so that Problem 7.3 in \cite{LaTho09} is still open.

Now let $\hat\xi:=\xi^\circ\otimes\lambda_1$ be
the extension of $\xi^\circ$ to $\bR^{d+1}$. Let $\pi$ be a preserving shift
for $\hat{\xi}$, that is,
\begin{align*}
\iint 1\{\tau(X^\circ,\hat\xi,s,u)\in B\times C\}\xi^\circ(ds)du
=\xi^\circ(B)\lambda_1(C),
\end{align*}
where, as before, $\tau$ is the allocation rule generated by $\pi$,
$B$ is a measurable subset of $\bR^d$ and $C$ is a measurable subset of $\bR$.
If $0<\lambda_1(C)<\infty$, then
\begin{align}\label{kernel5}
K_{\pi,C}(X^\circ,\xi^\circ,s,B):=\frac{1}{\lambda_1(C)}
\int 1\{\tau(X^\circ,\hat\xi,s,u)\in B\times C\}du
\end{align}
defines an invariant kernel $K_{\pi,C}$ that  preserves
$\xi^\circ$ in the sense of \eqref{kernel3}.

Theorem 8 says for a general $\xi^\circ$ that mass-stationarity of
$(X^\circ,\xi^\circ)$ is equivalent to 
distributional invariance of $(X^\circ,\xi^\circ)$ under the composition
of the transport kernels \eqref{kernel5} and \eqref{kernel1} (in this order).
The resulting composed kernel is not Markovian.

\end{document}